\documentclass[a4paper,12pt]{article}

\newenvironment{proof}{\noindent {\bf Proof:}}{\hfill $\Box$}

\newtheorem{lemma}{Lemma}

\newtheorem{remark}{Remark}

\usepackage{amsmath}
\usepackage{amssymb}
\usepackage{amsfonts}
\usepackage{graphicx}
\usepackage{subfigure}

\def\vol{\mathrm{vol}}

\def\trace{\mathrm{trace}}

\title{\bf Minimum volume semialgebraic sets for robust estimation}

\begin{document}

\author{Fabrizio Dabbene$^1$, Didier Henrion$^{2,3,4}$}

\footnotetext[1]{CNR-IEIIT; Corso Duca degli Abruzzi 123, Torino; Italy. {\tt fabrizio.dabbene@polito.it}}
\footnotetext[2]{CNRS, LAAS, 7 avenue du colonel Roche, F-31400 Toulouse; France. {\tt henrion@laas.fr}}
\footnotetext[3]{Universit\'e de Toulouse, LAAS, F-31400 Toulouse; France}
\footnotetext[4]{Faculty of Electrical Engineering, Czech Technical University in Prague,
Technick\'a 2, CZ-16626 Prague, Czech Republic}

\date{ \today}

\maketitle

\begin{abstract}
Motivated by problems of uncertainty propagation and robust estimation
we are interested in computing a polynomial sublevel set of fixed degree
and minimum volume that contains a given semialgebraic set $\bf K$.
At this level of generality this problem is not tractable, even though
it becomes convex e.g. when restricted to nonnegative homogeneous polynomials.
Our contribution is to describe and justify
a tractable $L^1$-norm or trace heuristic for this problem, relying upon
hierarchies of linear matrix inequality (LMI) relaxations when $\bf K$
is semialgebraic, and simplifying to linear programming (LP) when $\bf K$ is
a collection of samples, a discrete union of points.
\end{abstract}

\section{Introduction}

In this paper, we consider the problem of computing reliable approximations of a given set $\bf K\subset {\mathbb R}^{n}$.
The set $\bf K$ is assumed to have a complicated shape (e.g.\ nonconvex, non-connected), expressed in terms of semialgebraic conditions,
and we seek for approximations which should i) be easy computable and ii) have a simple description.

The problem of deriving reliable approximations of overly complicated sets 
by means of simpler geometrical shapes has a long history, and it arises
in many research fields related to optimization, system identification and control.

In particular, \textit{outer bounding sets}, i.e.\ approximations that are guaranteed to contain the set ${\bf K}$,
are widespread in the technical literature, and they find several applications in robust control and filtering.
For instance,  set-theoretic state estimators for uncertain discrete-time nonlinear dynamic systems have been proposed in \cite{AlBrCa:05,DuWaPo:01,ElGCal:01,ShaTu:97}.
These strategies adopt a set-membership approach \cite{GaTeVi:99,Schweppe:73}, and construct compact sets that are guaranteed to bound the systems states which are consistent with the measured output and the norm-bounded uncertainty.
Outer approximation also arise in the context of robust fault detection problems
(e.g., see \cite{IBPAG:09}) and of reachability analysis of nonlinear and/or hybrid systems \cite{HwStTo:03,KurVar:00}.

Similarly, \textit{inner} approximations  are employed in nonlinear programming \cite{NesNem:94}, in the solution of design centering problems \cite{WojVla:93} and for fixed-order controller design \cite{HenLou:12}.

Recently, the authors of \cite{DaLaSh:10} have proposed an approach  based on randomization, which constructs convex approximations of generic nonconvex sets which are neither inner nor outer, but they enjoy some specific probabilistic properties. In this context, an approximation is considered to be reliable if it contains ``most'' of the points in $\bf K$ with prescribed high probability. The key tool in this framework is the generation of random samples inside $\bf K$, and the construction of a convex set containing these samples.

In all the approaches listed above, several geometric figures have been adopted as approximating sets. The application of ellipsoidal sets to the state estimation problem has been introduced in the pioneering work \cite{Schweppe:73}
and used by many different authors from then on; see, for example, \cite{DuWaPo:01,ElGCal:01}. The use of polyhedrons was proposed
in \cite{KunLyc:85} to obtain an increased estimation accuracy, while zonotopes have been also recently studied in \cite{AlBrCa:05,GuNgZa:03}.

More recent works, like for instance \cite{BeViLa:07,HenLou:12,MaLaBo:05}, employ sets defined by semialgebraic conditions.
In particular, in \cite{MaLaBo:05} the authors use polynomial sum-of-squares (SOS) programming,
a particular class of SDP, to address the problem of fitting given data with
a convex polynomial, seen as a natural extension of quadratic polynomials
and ellipsoids. Convexity of the polynomial is ensured by enforcing that 
its Hessian is matrix SOS, and volume minimisation is indirectly
enforced by increasing the curvature of the polynomial.
In \cite{BeViLa:07} the authors propose moment relaxations for the separation
and covering problems with semialgebraic sets, thereby also extending the
classical ellipsoidal sets used in data fitting problems.

Our contribution is to extend further these works to cope with volume minimization
of arbitrary (e.g. non-convex, non-connected, higher degree) semialgebraic
sets containing a given semialgebraic set (e.g. a union of points).
Since there is no analytic formula for the volume of a semialgebraic set,
in terms of the coefficients of the polynomials defining the set, we have
no hope to solve this optimization problem globally. Instead,
we describe and justify analytically and geometrically
a computationally tractable heuristic based on $L^1$-norm or trace
minimization.

\section{Problem statement}

Given a compact basic semialgebraic set
\[
{\bf K}:=\{x \in {\mathbb R}^n : g_i(x) \geq 0, \:i=1,2,\ldots,m\}
\]
where $g_i(x)$ are real multivariate polynomials,
we want to compute a polynomial sublevel set
\[
{\bf V}(q):=\{x \in {\mathbb R}^n : q(x) \leq 1\} \supset {\bf K}
\]
of minimum volume that contains $\bf K$. Set ${\bf V}(q)$ is
modeled by a polynomial $q$ belonging to ${\bf P}_d$,
the vector space of multivariate real polynomials of degree less than or equal to $d$.
In other words, we want to solve the following optimization problem:
\begin{equation}\label{opt}
\begin{array}{ll}
\inf_{q \in {\bf P}_d} & \vol\:{\bf V}(q) \\
\mathrm{s.t.} & {\bf K} \subset {\bf V}(q). \\
\end{array}
\end{equation}
In the above problem
\[
\vol\:{\bf V}(q):=\int_{{\bf V}(q)} dx = \int_{{\mathbb R}^n} I_{{\bf V}(q)}(x) dx
\]
is the volume or Lebesgue measure of set ${\bf V}(q)$ and $I_{\bf X}(x)$ is the indicator
function, equal to one when $x \in {\bf X}$ and zero otherwise.

Note that, typically, the set $\bf K$ has a complicated description in terms of polynomials
$g_i$ (e.g. coming from physical measurements and/or estimations)
and set ${\bf V}(q)$ has a simple description (in the sense that the degree of $q$
is small, say less than 10). Minimization
of the volume of ${\bf V}(q)$ means that we want ${\bf V}(q)$ to capture most of
the geometric features of $\bf K$. 

If $\bf K$ is convex and $q$ is quadratic, then the infimum of problem (\ref{opt})
is attained, and there is a unique (convex) ellipsoid ${\bf V}(q)$ of minimum volume that
contains $\bf K$, called L\"owner-John ellipsoid.
It can be computed by convex optimization, see e.g. \cite[\S 4.9]{BenNem:01}.

In general, without convexity assumptions on $\bf K$, a solution to problem (\ref{opt})
is not unique. There is also no guarantee that the computed set ${\bf V}(q)$ is convex.
Optimization problem (\ref{opt}) is nonlinear and semi-infinite, in the sense
that we optimize over the finite-dimensional vector space ${\bf P}_d$
but subject to an infinite number of constraints (to cope with set inclusions).

If we denote by $\pi_d(x)$ a (column vector) basis of monomials of
degree up to $d$, we can write $q(x) = \pi^T_d(x) q$ where $q$ is a vector
of coefficients of given size. In vector space ${\bf P}_d$, the set
${\bf Q}:=\{q : \pi^T_d(x)q \leq 1\}$ is by definition the polar of the bounded set
$\{\pi_d(x) : \|x\| \leq R\}$ where $R>0$ is a constant chosen sufficiently
large so that all vectors $x \in {\bf V}(q)$ have norm less than $R$. As the polar
of a compact set whose interior contains the origin, set $\bf Q$ is compact, see e.g.\ \cite{Lay:92}. 
It follows that the feasible set of problem (\ref{opt})
is compact, and since the objective function is continuous, the infimum
in problem (\ref{opt}) is attained. 

\section{Convex conic formulation}

In this section it is assumed that $q$ is a nonnegative homogeneous polynomial, or form, of
even degree $d=2\delta$ in $n$ variables. Under this restriction, in \cite[Lemma 2.4]{Lasserre:11}
it is proved that the volume function
\[
q \mapsto \vol\:{\bf V}(q)
\]
is convex in $q$. The proof of this statement relies on the striking observation \cite{MorSha:10} that
\[
\vol\:{\bf V}(q) = C_d \int_{{\mathbb R}^n} e^{-q(x)}dx
\]
where $C_d$ is a constant depending only on $d$.
Note also that boundedness of ${\bf V}(q)$ implies that $q$ is nonnegative,
since if there is a point $x_0 \in {\mathbb R}^n$ such that $q(x_0)<0$, and hence $x_0 \in {\bf V}(q)$,
then by homogeneity of $q$ it follows that $q(\lambda x_0)=\lambda^{2\delta} q(x_0)<0$ for all $\lambda$
and hence $\lambda x_0 \in {\bf V}(q)$ for all $\lambda$ which contradicts
boundedness of ${\bf V}(q)$. This implies that problem (\ref{opt}), once restricted to nonnegative forms,
is a convex optimization problem.

In \cite[Lemma 2.4]{Lasserre:11} explicit expressions are given for the
first and second order derivatives of the volume function, in terms
of the moments
\begin{equation}\label{mom}
\int_{{\mathbb R}^n} x^{\alpha} e^{-q(x)}dx
\end{equation}
for $\alpha \in {\mathbb N}^n$, $|\alpha|\leq 2d$.
In an iterative algorithm solving convex problem (\ref{opt}),
one should then be able to compute repeatedly
and quickly integrals of this kind, arguably a difficult task.

When $q$ is not homogeneous, we do not know under which
conditions on $q$ the volume function ${\bf V}(q)$ is convex in $q$.

Motivated by these considerations, in the remainder of the paper we propose a simpler approach
to solving problem (\ref{opt}) which is not restricted to
forms, and which does not require the potentially intricate
numerical computation of moments (\ref{mom}).
Our approach is however only a heuristic, in the sense that
we do not provide guarantees of solving problem (\ref{opt})
globally.

\section{$L^1$-norm minimization}

Let us write ${\bf V}(q)$ as a polynomial superlevel set
\[
{\bf U}(p):={\bf V}(q)=\{x \in {\mathbb R}^n : p(x):=2-q(x) \geq 1\}.
\]
with polynomial
\begin{equation}\label{gram}
p(x)=\pi^T_\delta(x) P \pi_\delta(x)
\end{equation}
expressed as a quadratic form in a given (column vector) basis $\pi_{\delta}(x)$ of monomials of
degree up to $\delta:=\lceil\frac{d}{2}\rceil$, with symmetric Gram matrix $P$.
Then, optimization problem (\ref{opt}) reads
\begin{equation}\label{optp}
\begin{array}{rcll}
v^*_d & := & \min_{p \in {\bf P}_d} & \vol\:{\bf U}(p) \\
&& \mathrm{s.t.} & {\bf K} \subset {\bf U}(p). \\
\end{array}
\end{equation}
Note that in problem (\ref{optp}) we can indifferently optimize over coefficients of $p$
or coefficients of matrix $P$, since they are related linearly
via constraint (\ref{gram}).

Since $\bf K$ is compact by assumption and ${\bf U}(p)$ is compact for problem (\ref{optp})
to have a finite minimum, we suppose that we are given a compact semialgebraic set
\[
{\bf B}:=\{x \in {\mathbb R}^n : b_i(x) \geq 0, \: i=1,2,\ldots,m_b \}
\]
such that ${\bf U}(p) \subset {\bf B}$ and hence
\[
{\bf U}(p)=\{x \in {\bf B} : p(x) \geq 1\}.
\]
The particular choice of polynomials $b_i$ will be specified later on.
Now, observe that by definition
\[
p(x) \geq I_{{\bf U}(p)}(x) \:\:\mathrm{on}\:\: {\mathbb R}^n
\]
and hence, integrating both sides we get
\[
\int_{\bf B} p(x)dx \geq \int_{\bf B} I_{{\bf U}(p)}(x)dx = \vol\:{\bf U}(p),
\]
an inequality known as Chebyshev's inequality, 
widely used in probability, see e.g. \cite[\S 2.4.9]{AshDol:00}.
If polynomial $p$ is nonnegative on $\bf B$
then the above left-hand side is the $L^1$-norm of $p$, and
the inequality becomes
\begin{equation}\label{cheb}
\|p\|_1 \geq \vol\:{\bf U}(p).
\end{equation}

Now consider the following $L^1$-norm minimization problem:
\begin{equation}\label{l1}
\begin{array}{rcll}
w^*_d & := & \min_{p \in {\bf P}_d} & \|p\|_1 \\
&& \mathrm{s.t.} & p \geq 0 \:\:\mathrm{on}\:\: {\bf B} \\
&&& p \geq 1 \:\:\mathrm{on}\:\: {\bf K}.
\end{array}
\end{equation}

\begin{lemma}\label{cvg}
The minimum of problem (\ref{l1}) monotonically
converges from above to the minimum of problem (\ref{optp}), i.e.
$w^*_{d-1} \geq w^*_d \geq v^*_d$ for all $d$, and
$\lim_{d\to\infty} w^*_d = \lim_{d\to\infty} v^*_d$.
\end{lemma}

\begin{proof}
The graph of polynomial $p$ lies above $I_{\bf K}$, the indicator
function of set $\bf K$, while being nonnegative on $\bf B$,
so minimizing the $L^1$-norm of $p$ on $\bf B$ yields
an upper approximation of $I_{\bf K}$. Monotonicity of
the sequence $w^*_d$ follows immediately since
polynomials of degree $d+1$ include polynomials
of degree $d$. When its degree increases, $p$ converges in $L^1$-norm
to $I_{\bf K}$, hence $\|p\|_1$ converges to $\vol\:{\bf K}$.
The convergence is pointwise almost everywhere,
and almost uniform, but not uniform since $I_{\bf K}$
is discontinuous on $\bf B$.
\end{proof}

Note that this $L^1$-norm minimization
approach was originally proposed in \cite{HeLaSa:09}
to compute numerically the volume and moments
of a semialgebraic set.

\section{Trace minimization}

In this section we give a geometric interpretation
of problem (\ref{l1}).
First note that the objective function reads
\begin{eqnarray*}
\|p\|_1 &=& \int_{\bf B} p(x)dx = \int_{\bf B} \pi^T_\delta(x)P\pi_\delta(x)dx\\
&=& \trace\left(P\int_{\bf B} \pi_\delta(x)\pi^T_\delta(x)dx\right)
= \trace\:PM
\end{eqnarray*}
where
\[
M:=\int_{\bf B} \pi_\delta(x)\pi^T_\delta(x)dx
\]
is the matrix of moments of the Lebesgue measure on $\bf B$ in basis $\pi_\delta(x)$.
In equation (\ref{gram}) if the basis is chosen such that its entries
are orthonormal w.r.t. the (scalar product induced by the)
Lebesgue measure on $\bf B$, then $M$ is the
identity matrix and inequality (\ref{cheb}) becomes
\begin{equation}\label{trace}
\trace\:P \geq \vol\:{\bf U}(p)
\end{equation}
which indicates that, under the above constraints,
minimizing the trace of the Gram matrix $P$
entails minimizing the volume of ${\bf U}(p)$.

The choice of polynomials $b_i$ in the definition of the bounding
set $\bf B$ should be such that the objective function in problem
(\ref{l1}) is easy to compute. If
\[
p(x)=\pi^T_d(x) p = \sum_\alpha p_\alpha [\pi_d(x)]_\alpha
\]
then
\[
\int_{\bf B} p(x)dx = \sum_\alpha p_\alpha \int_{\bf B} [\pi_d(x)]_{\alpha}dx
= \sum_\alpha p_\alpha y_\alpha
\]
and we should be able to compute easily the moments
\[
y_\alpha := \int_{\bf B} [\pi_d(x)]_{\alpha} dx
\]
of the Lebesgue measure on $\bf B$ w.r.t. basis $\pi_d(x)$.
This is the case e.g. if $B$ is a box.
\vskip .1in

\begin{remark}[Minimum trace heuristic for ellipsoids]
Note that, in the case of quadratic polynomials, i.e. $d=2$, we retrieve
the classical trace heuristic used for volume minimization,
see e.g. \cite{DuPoWa:96}. If $B=[-1,1]^n$ then the basis $\pi_1(x)=\frac{\sqrt{6}}{2}x$
is orthonormal w.r.t. the Lebesgue measure on $B$ and
$\|p\|_1 = \frac{3}{2}\trace\:P$. The constraint that $p$ is
nonnegative on ${\bf B}$ implies that the curvature of the boundary
of ${\bf U}(p)$ is nonnegative, hence that ${\bf U}(p)$ is convex.
\end{remark}
\vskip .1in

In \cite{MaLaBo:05}, the authors restricted the search to convex
polynomial sublevels ${\bf U}(p)={\bf V}(q)=\{x : q(x) \leq 1\}$ by enforcing positive
semidefiniteness of the Gram matrix of the quadratic Hessian form of $q$.
They proposed to maximize (the logarithm of) the determinant of the
Gram matrix, justifying this choice by explaining that this
increases the curvature of the polynomial sublevel set along all
directions, and hence minimizes the volume. Supported by Lemma \ref{cvg}
and the above discussion, we came to the consistent conclusion that
minimizing the trace of the Gram matrix of $p$, that is,
maximizing the trace of the Gram matrix of $q$, is a relevant heuristic
for volume minimization. Note however that in our approach
\textit{we do not enforce convexity} of ${\bf U}(p)$.

\section{Handling constraints}

If ${\bf K}=\{x : g_i(x) \geq 0, \:i=1,2,\ldots,m\}$ is a general
semialgebraic set then we must ensure that polynomial $p-1$ is nonegative
on $\bf K$, and for this we can use Putinar's Positivstellensatz and
a hierarchy of finite-dimensional convex LMI relaxations which are
linear in the coefficients of $p$. More specifically, we write
$p-1= r_0 + \sum_i r_i g_i$ where $r_0, r_1, \ldots, r_m$ are
polynomial sum-of-squares of given degree, to be found.
For each fixed degree, the problem of finding such polynomials
is an LMI, see e.g. \cite[Section 3.2]{Lasserre:11}.
The constraint that $p$ is nonnegative on $\bf B$ can be ensured
similarly.

A particularly interesting case is when  $\bf K$ is a discrete set, i.e. a union of points $x_i \in {\mathbb R}^n$,
$i=1,\ldots,N$. Indeed, in this case the inclusion constraint ${\bf K} \subset {\bf U}(p)$ is equivalent
to a finite number of inequalities $p(x_i) \geq 1$,
$i=1,\ldots,N$ which are \textit{linear} in the coefficients of $p$.
Similarly, the constraint that $p$ is nonnegative on $\bf B$ can
be handled by linear inequalities $p(x_j) \geq 0$ enforced at
a dense grid of points $x_j \in {\bf B}$, $j=1,\ldots,M$, for $M$
sufficiently large. Note that, with this pure linear programming (LP) approach,
it is not guaranteed that $p$ is nonnegative on $\bf B$, but what matters
primarily is that ${\bf K} \subset {\bf U}(p)$, which is indeed guaranteed.
This purely LP formulation allows to deal with problems with rather large $N$.

\section{Examples}

We illustrate the proposed approach for the case when $\bf K$ is a discrete set,
a union of points of ${\mathbb R}^n$. For our numerical examples,
we have used the YALMIP interface for Matlab to model the
LMI optimization problem (\ref{opt}), and the SDP solver SeDuMi
to solve numerically the problem. Since the degrees of
the semialgebraic sets we compute are typically low (say less than 20),
we did not attempt to use appropriate polynomial bases (e.g. Chebyshev
polynomials) to improve the quality and
resolution of the optimization problems, see however \cite{HeLaSa:09}
for a discussion on these numerical matters.

\subsection{Line ($n=1$)}

\begin{figure}[h!]
\centering
\includegraphics[width=0.48\textwidth]{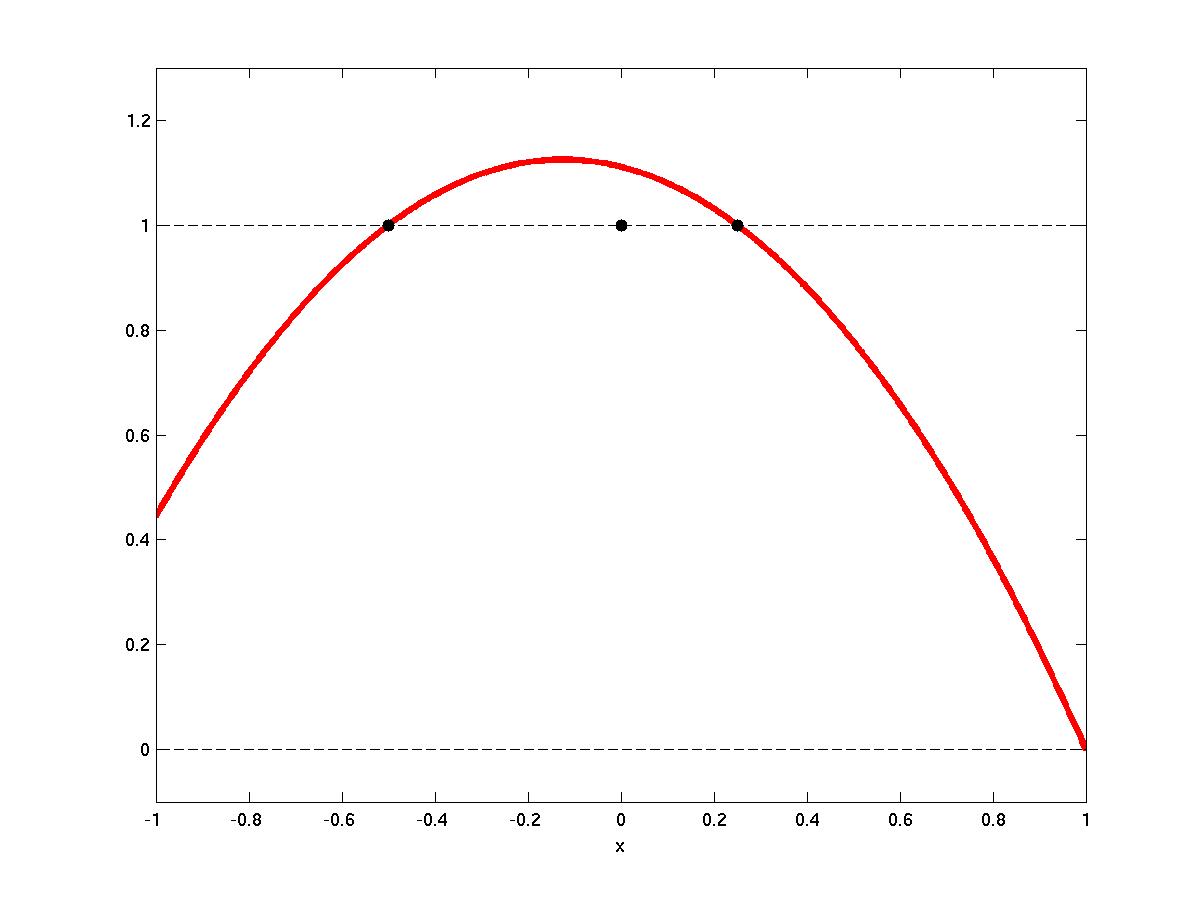}
\includegraphics[width=0.48\textwidth]{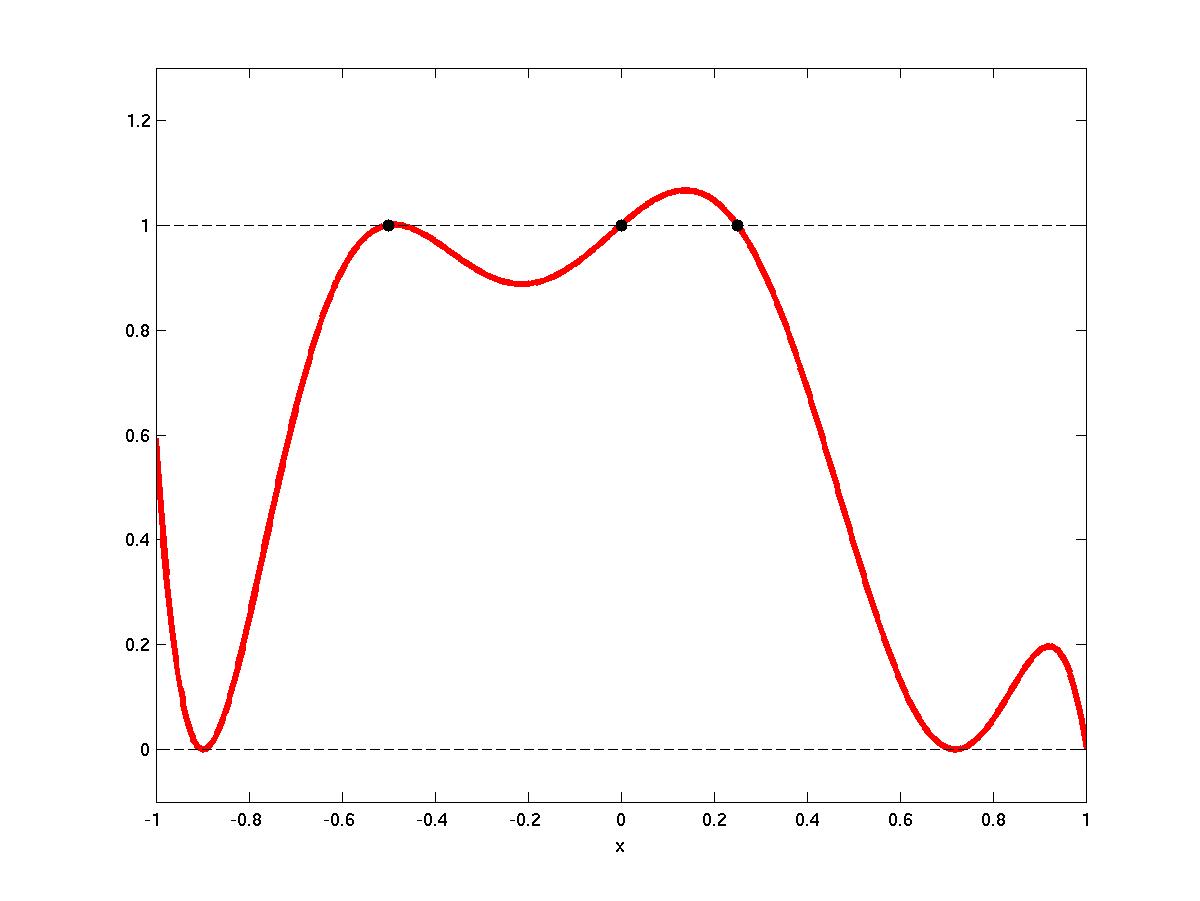}\\
\includegraphics[width=0.48\textwidth]{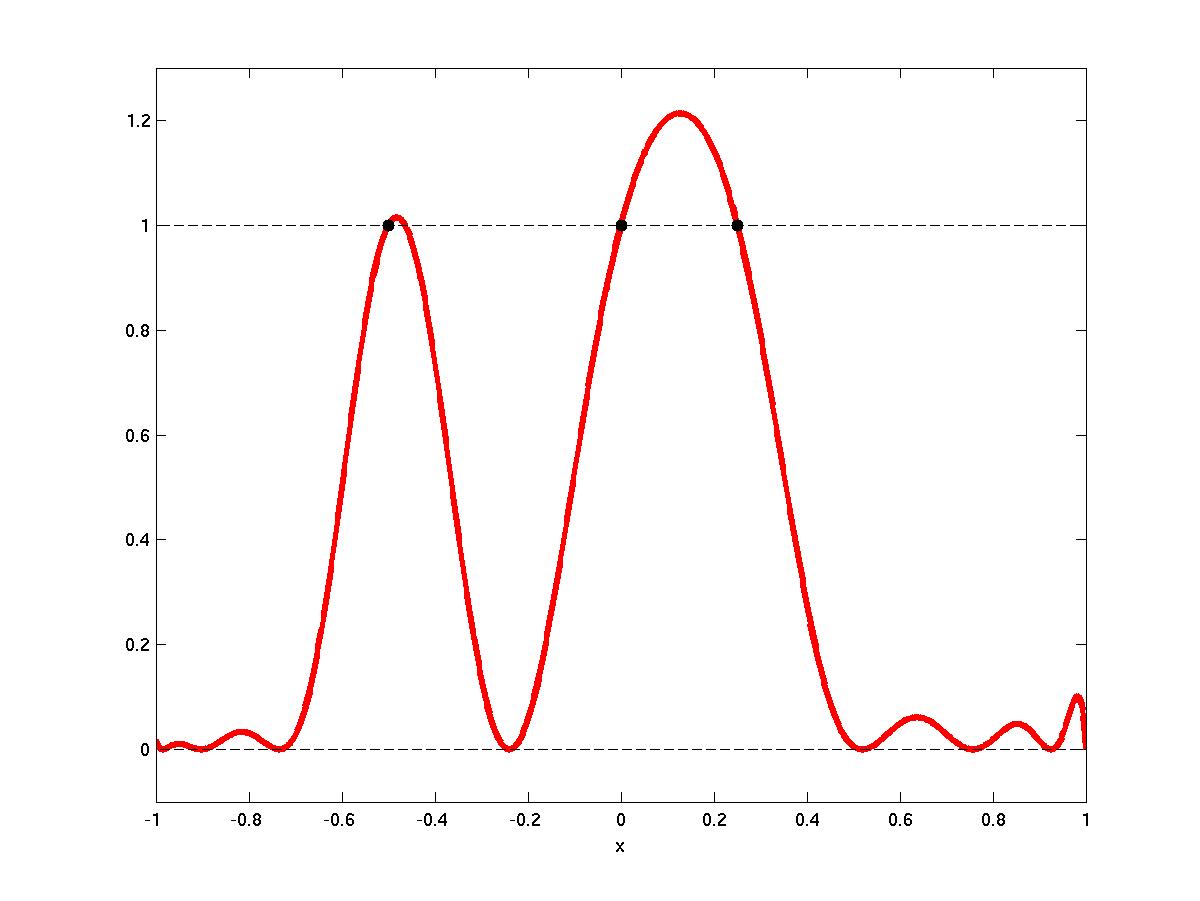}
\includegraphics[width=0.48\textwidth]{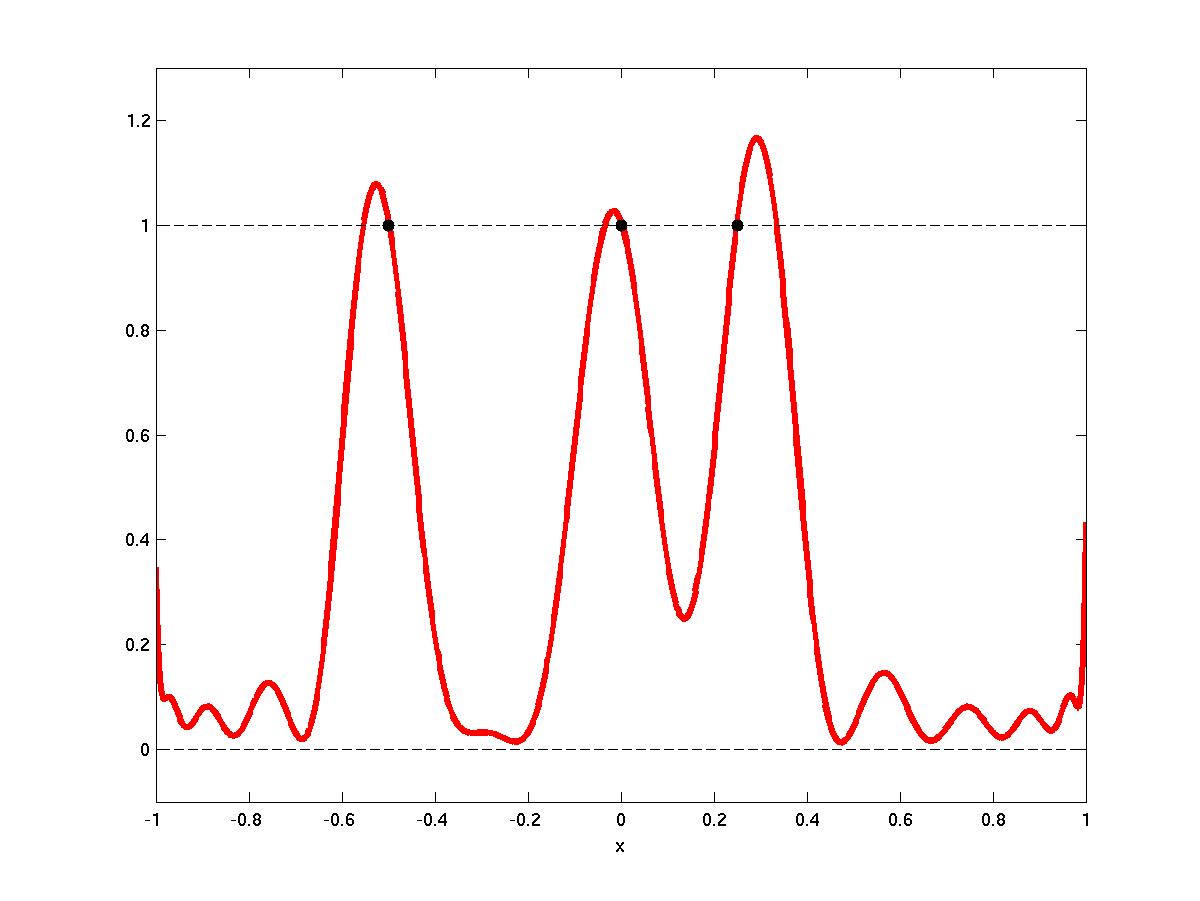}\\
\caption{Minimum $L^1$-norm polynomials $p$ (red) such that $p\geq 0$ on $[-1,\:1]$ and
$p\geq 1$ at the 3 points (black), for degree 2 (upper left), 7 (upper right), 17 (lower left)
and 26 (lower right).\label{line}}
\end{figure}

To illustrate the behavior of the proposed optimization procedure, we first consider ${\bf B}=[-1,\:1]$ and $K=\{-\frac{1}{2},0,\frac{1}{4}\}$.
On Figure \ref{line} we represent the solutions $p$ of degrees 2, 7, 17 and 26
of minimization problem (\ref{l1}). We observe that the superlevel set
${\bf U}(p)=\{x \in {\bf B} : p(x) \geq 1\}$ is simply connected for degree 2,
doubly connected for degree 7, and triply connected for degrees 17 and 26.
Note that, as the degree increases of $p$, the length of the intervals for which
$p(x)>1$ tends to zero. 
This is consistent with the fact that the volume of a finite set is zero.
We also observe that polynomial $p$ shows increasing oscillations
for increasing degrees, a typical feature of such discontinuous function
approximations.

\subsection{Plane ($n=2$)}

\begin{figure}[h!]
\centering
\includegraphics[width=0.48\textwidth]{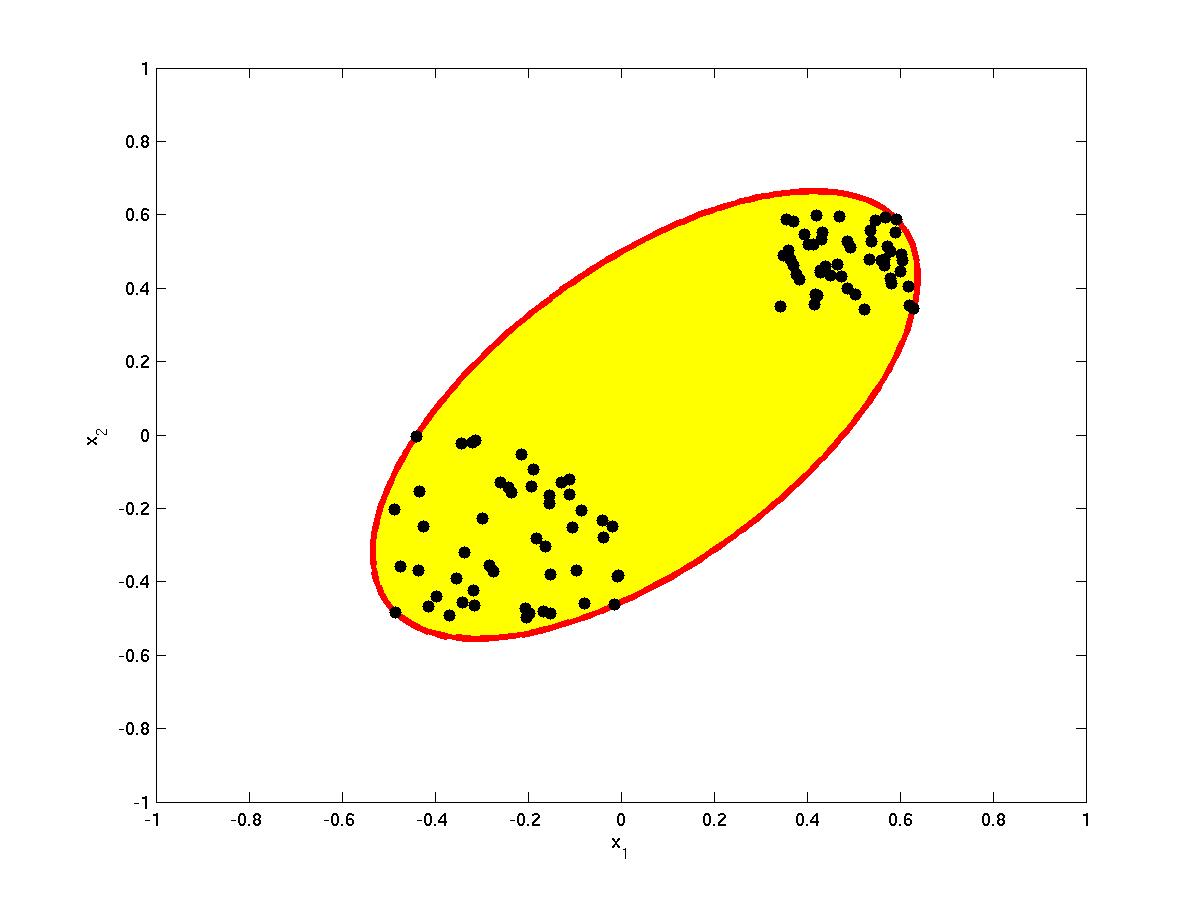}
\includegraphics[width=0.48\textwidth]{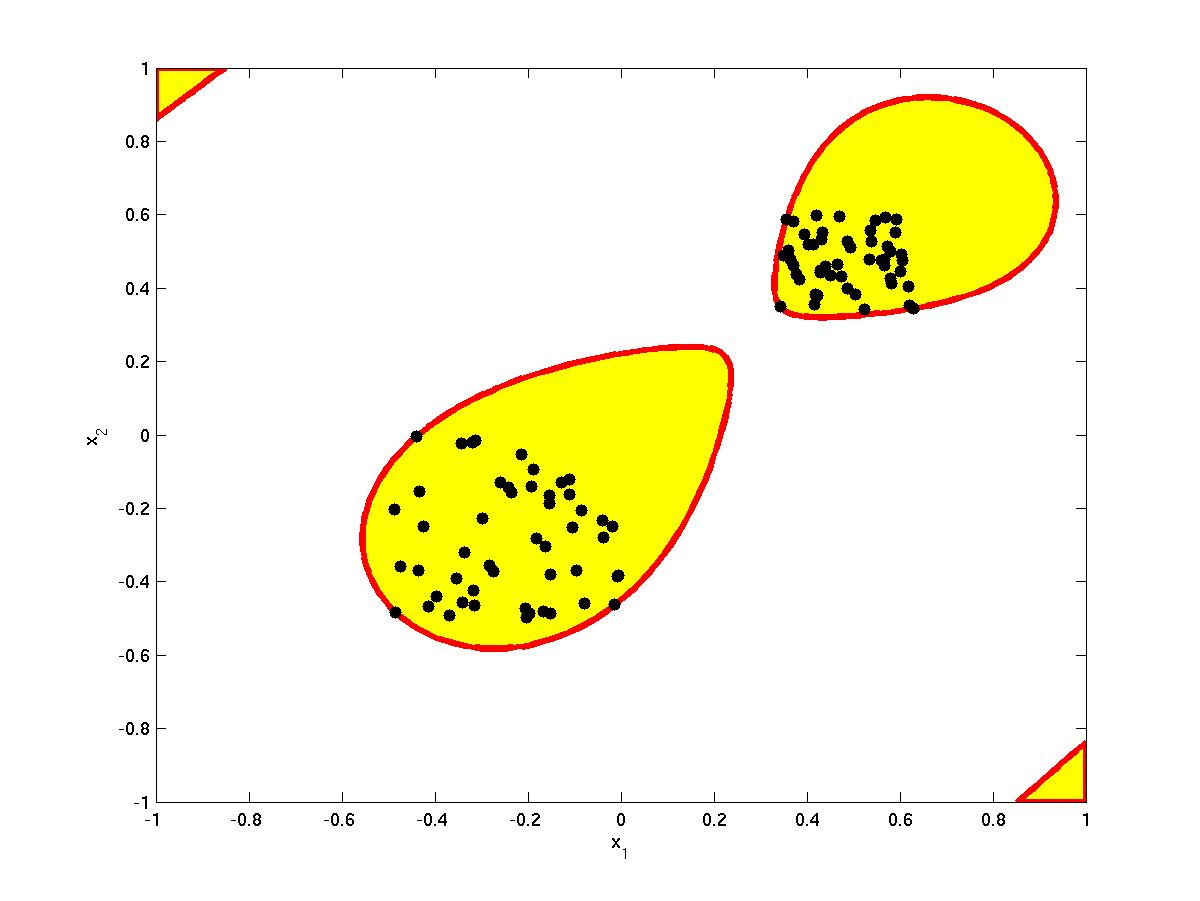}\\
\includegraphics[width=0.48\textwidth]{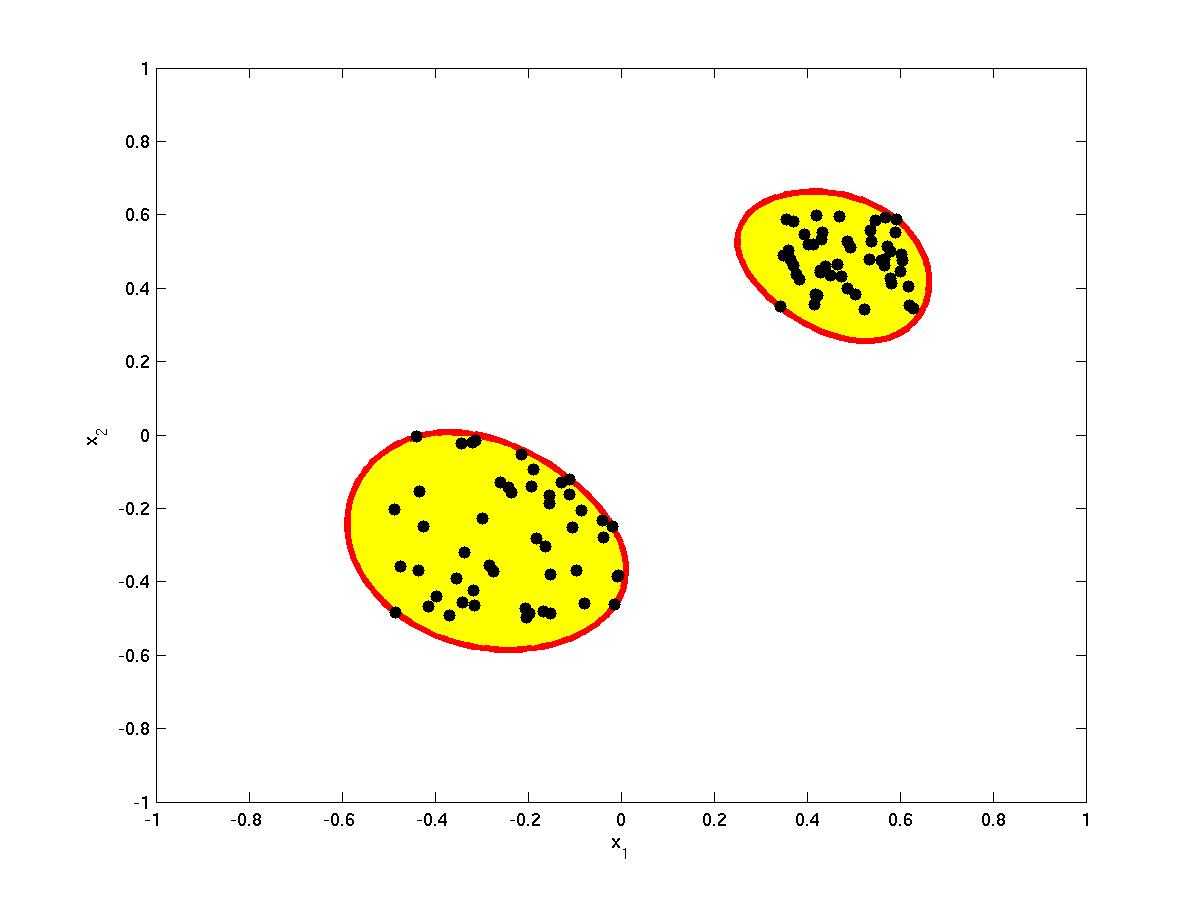}
\includegraphics[width=0.48\textwidth]{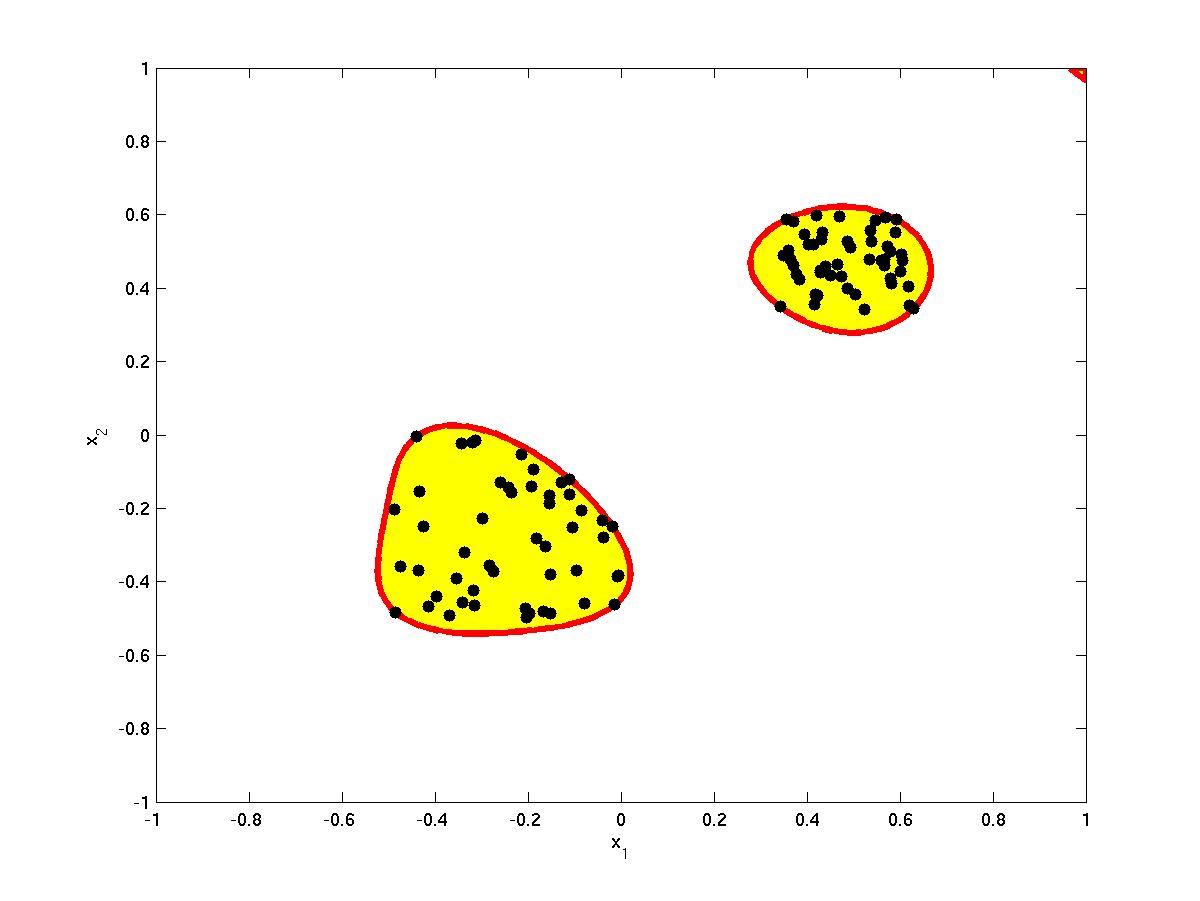}\\
\caption{Minimum $L^1$-norm polynomials $p$ (red) and low surface semialgebraic sets (yellow)
such that $p\geq 0$ on $[-1,\:1]^2$ and $p\geq 1$ at 100 points (black), for degree 2 (upper left),
5 (upper right), 9 (lower left) and 14 (lower right).\label{plane}}
\end{figure}

In ${\bf B}=[-1,\:1]^2$ we consider two clouds of 50 points each, i.e. $N=100$.
On Figure \ref{plane} we represent the solutions $p$ of degrees 2, 5, 9 and 14
of minimization problem (\ref{l1}). Here too we observe that increasing the
degree of $p$ allows to disconnect set ${\bf U}(p)$.
The side effects near the border of ${\bf B}$ on the top right figure
can be removed by enlarging the bounding set ${\bf B}$.

\subsection{Space ($n=3$)}

\begin{figure}[h!]
\centering
\includegraphics[width=0.35\textwidth]{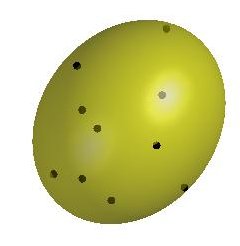} 
\includegraphics[width=0.3\textwidth]{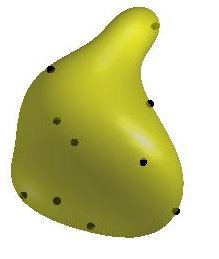} 
\includegraphics[width=0.31\textwidth]{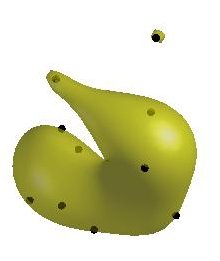}
\caption{Including the same 10 space points (black) in low volume semialgebraic sets (yellow)
of degree 4 (left), 10 (center), 14 (right).\label{space}}
\end{figure}

In ${\bf B}=[-1,\:1]^3$ we consider $N=10$ points.
The solutions $p$ of degrees 4, 10, and 14 of minimization problem (\ref{l1})
is depicted in Figure \ref{space}. Here too we observe that increasing the
degree of $p$ allows to capture point clusters in distinct
connected components. 

It should be pointed out that all the semialgebraic
sets in the previous examples were computed in a few seconds of CPU time on a standard PC.

\section{Conclusion}
In this paper, we proposed a simple technique for approximating a given set of complicated shape by means of the polynomial sub level set of minimal size that contains it.
The proposed approximation has been shown to be convex in the case the polynomial is assumed to be homogeneous. Then, a tractable relaxation based on Chebychev's inequality has been introduced. Interestingly, this relaxation reduces to the classical trace minimization heuristic in the case of quadratic polynomials, thus indirectly providing an intuitive explanation to this widely used criterion.

Ongoing research is devoted to utilize the proposed approximation to construct new classes of set-theoretic filters, in the spirit of the works \cite{AlBrCa:05,ElGCal:01}.

\section*{Acknowlegdments}

This work was partially supported by a Bilateral Project CNRS-CNR.

\end{document}